\documentclass[11pt]{amsart}
\usepackage[margin=2cm]{geometry}
\usepackage{color,xcolor}
\usepackage{amsfonts,amsmath,amsthm,amssymb,tikz-cd,enumitem}
\usepackage{mathabx}
\usepackage{verbatim}   
\usepackage{fullpage}
\usepackage{hyperref}
\hypersetup{
    colorlinks=true, 
    linkcolor=blue,
    filecolor=magenta,
    urlcolor=cyan,
    citecolor=cyan}
\usepackage[nameinlink,capitalize]{cleveref}
\usepackage{tikz,tikz-cd}
\usetikzlibrary{arrows,matrix}
\theoremstyle{definition}

\newtheorem{thm}{Theorem}[section]
\newtheorem{example}[thm]{Example}
\newtheorem{defn}[thm]{Definition}

\newtheorem{lemma}[thm]{Lemma}
\newtheorem{cor}[thm]{Corollary}
\newtheorem{rmk}[thm]{Remark}

\newtheorem{mainthm}{Theorem}

\newcommand{\uline}[1]{\underline{#1}}

\newcommand{\fbp}[1]{\left[#1 \right]}

\def\Hom{\operatorname{Hom}}

\def\fm{\mathfrak{m}}

\def\fp{\mathfrak{p}}

\newcommand{\hsl}{\operatorname{HSL}}

\newcommand{\mbn}{\mathbf{N}}
\newcommand{\mbz}{\mathbf{Z}}
\newcommand{\ann}{\operatorname{Ann}}
\newcommand{\mfm}{\mathfrak{m}}
\newcommand{\mfn}{\mathfrak{n}}
\newcommand{\mfb}{\mathfrak{b}}
\newcommand{\mfp}{\mathfrak{p}}
\newcommand{\mfq}{\mathfrak{q}}
\newcommand{\mcw}{\mathcal{W}}

\newcommand{\on}[1]{\operatorname{#1}}
\newcommand{\soc}{\on{Soc}}

\newcommand{\im}{\operatorname{im}}

\newcommand{\depth}{\operatorname{depth}}
\newcommand{\gdp}{\operatorname{g-depth}}
\newcommand{\fdp}{F\textrm{-depth}}
\newcommand{\gfdp}{\textrm{g}F\textrm{-depth}}
\newcommand{\spec}{\operatorname{Spec}}

\def\NN{\mathbf{N}}
\def\ZZ{\mathbf{Z}}

\def\cG{\mathcal{G}}
\def\cR{\mathcal{R}}

\def\ux{\underline{x}}

\newcommand{\Arg}{\rule{1ex}{1pt}}

\newcommand{\sbt}{\,\begin{picture}(-1,1)(-1,-2)\circle*{2}\end{picture}\ }

\colorlet{DG}{green!50!black}
\colorlet{DB}{red!50!black}

\def\LEM[#1]{\footnote{ {\color{DG} LEM: #1  }  }}
\def\KLM[#1]{\footnote{ {\color{purple} KLM: #1  }  }}
\def\AC[#1]{\footnote{ {\color{DB} AC: #1  }  }}

\tikzset{
    labl/.style={anchor=south, rotate=270, inner sep=.5mm}
}

\title{Rees algebras and generalized depth-like conditions in prime characteristic}

\author{Alessandra Costantini, Kyle Maddox, and Lance Edward Miller}

\date{August 19, 2023}

\begin{document}

\maketitle

\begin{abstract}
In this article we address a question concerning nilpotent Frobenius actions on Rees algebras and associated graded rings. We prove a nilpotent analog of a theorem of Huneke for Cohen-Macaulay singularities. This is achieved by introducing a depth-like invariant which captures as special cases Lyubeznik's F-depth and the generalized F-depth from Maddox-Miller and is related to the generalized depth with respect to an ideal. We also describe several properties of this new invariant and identify a class of regular elements for which weak F-nilpotence deforms.
\end{abstract}

\section{Introduction}

This article is a continuation of the study of singularities of local rings in positive characteristic via nilpotent Frobenius actions on local cohomology. Introduced in \cite{BB05}, the classes of $F$-nilpotent singularities and its variants have gathered significant recent attention  \cite{KMPS,Mad19,MM,MP22,PQ19,Quy19,ST17}. Critical tools for studying these singularities are Lyubeznik's $F$-depth, which tracks the smallest cohomological index of a non-nilpotent local cohomology module of a local ring, and variations of this invariant as introduced in \cite{MM}. We show in this article that $F$-depth, and indeed many other depth-like invariants, can be interpreted as different instances of a common generalized invariant. This is both predicted by, and strengthens, growing analogies between rings with nilpotent type singularities and those with generalized Cohen-Macaulay singularities.

A local ring $(R,\fm)$ is called {\it weakly $F$-nilpotent} provided that all lower local cohomology modules $H_\fm^i(R)$ are nilpotent under their canonical Frobenius action. These can be seen as analogous to Cohen-Macaulay singularities, replacing the vanishing condition defining Cohen-Macaulayness with nilpotence. Likewise, {\it generalized weakly $F$-nilpotent singularities} were identified in \cite{Mad19} as analogs of generalized Cohen-Macaulay singularities. All of these singularities can be equivalently described as rings for which natural depth-like invariants are as large as possible \cite{MM}. Additionally, building on the framework introduced in \cite{MM}, these notions are considered naturally in graded settings as well.

With an eye towards future applications, we draw inspiration from the rich study of the Cohen-Macaulay property of Rees algebras and associated graded rings, where similar invariants were considered \cite{HuGrI,TI89,HM94}. As such, for an ideal $J\subset R$ we introduce a new invariant, the {\bf generalized $F$-depth with respect to an ideal $J$} \[ \gfdp_J(R) := \inf \{j \in \mbn \mid (J^N H^j_\mfm(R),\rho) \text{ is not nilpotent for any } N\in \mbn \}.\]

In various specializations, this invariant recovers Lyubeznik's $F$-depth \cite{Lyu06} and the generalized $F$-depth introduced in \cite{MM}. Also, as zero modules are obviously nilpotent, it generalizes a similar invariant introduced in \cite{HM94} in the context of generalized Cohen-Macaulay rings. In this paper, we embark on a detailed study of the foundational properties of this new invariant and discuss various applications. 
Perhaps the most tantalizing is an analog of a theorem of Huneke \cite[Prop.~1.1]{HuGrI} which states that, for an ideal $I$ of positive height in a Cohen-Macaulay ring, the Cohen-Macaulay property transfers from the Rees algebra $\cR$ to the associated graded ring $\cG$ of $I$. More generally, we prove the following main theorem.

\begin{mainthm}[{\rm see \Cref{thm: bounding gfdp_J of G}}]
Let $S$ be a local ring of dimension $d\ge 1$ and fix $I\subset S$ an ideal of positive grade. Set $\cR = S[It]$ and $\cG = \bigoplus_{i \ge 0} I^i/I^{i+1}$ the associated graded ring of $I$. Let $J_0\subset S$ be an ideal. For $J$ a homogeneous ideal of $\cR$ such that $[J]_0=J_0$, $$\gfdp_{J\cG}(\cG) \ge \min \{\gfdp_{J_0} (S) ,\gfdp_J (\cR) -1\}.$$
In particular, if $S$ and $\cR$ are (generalized) weakly $F$-nilpotent, so is $\cG$.
\end{mainthm}

This can be seen as a complementary result to other work examining $F$-singularities for Rees algebras, specifically $F$-rational singularities \cite{HWY02a,KK}, $F$-regular singularities \cite{HWY02b},  or $F$-pure singularities \cite{dSMNB}.

The literature on Cohen-Macaulay singularities \cite{TI89,HM94} suggests that a deeper understanding of generalized $F$-depth is required in order to conversely transfer nilpotent-type singularities from the associated graded ring to the Rees algebra.  With this future program in mind, an important finding is the following local description of generalized $F$-depth, which we view as an analog to \cite[Prop.~2.1]{HM94}.

\begin{mainthm}[{\rm see \Cref{thm: can check gfdp_J(R) locally on D(J)}}]
Suppose $R$ is complete, $F$-finite, and equidimensional. For ideal $J\subset R$ \[
\gfdp_J (R) = \inf\{\fdp (R_\mfp) + \dim(R/\mfp) \mid J\not \subset \mfp\}.
\] In particular, $R$ is generalized weakly $F$-nilpotent with respect to $J$ if and only if the Zariski open set $D(J)=\{ \mfp \in \spec R \mid J\not\subset \mfp\}$ is inside the weakly $F$-nilpotent locus of $R$.
\end{mainthm}

An obstacle that makes the theory of $F$-depth and its generalizations more difficult to study than ordinary depth is that, for a regular element $x \in R$, we need not have $\fdp (R/xR) = \fdp (R) -1$; in fact, we exhibit explicit examples of rings where this equality fails for all $x\in R$. To continue to extend results known for Cohen-Macaulayness in the nilpotent setting, we address this obstacle by introducing and studying a class of regular elements called {\bf $F$-depth elements} for which the equality above is satisfied (see \cref{defn: Fdepth elements}). In the presence of such elements (e.g., when $R$ is a weakly $F$-nilpotent local ring, \Cref{rmk: every reg elt is F-depth for wFn}), classical results known for Cohen-Macaulay singularities can be proven for nilpotent-type singularities as well; in particular, compare the following result with \cite[Lem.~2.5]{HM94}. 

\begin{mainthm}({\rm see \Cref{thm: generalized F-depth at irrelevant ideal}})
If $H^j_\mfm(R)$ is nilpotent in all negative degrees for $0 \le j < t$, then $\gfdp_{R^+} (R) \ge t$. The converse holds if $R$ admits a homogeneous $F$-depth element on $R$ of positive degree. 
\end{mainthm}

The article is organized as follows. In Section 2 we introduce all needed background results and preliminaries. In Section 3, we define the fundamental invariant $\gfdp_J (R)$ and study its basic properties. In Section 4, we study the applications to Rees algebras and associated graded rings. Finally, in Section 5, we introduce $F$-depth elements, and examine their properties and applications. \\

\noindent {\bf Acknowledgements:} We thank Ian Aberbach and Thomas Marley for helpful discussions relevant to this work. We are also grateful to the anonymous referees for their insightful comments which have lead to significant improvements of the paper.

\section{Preliminaries}

In this paper, all rings will be assumed to be commutative and Noetherian. By convention, we use $\bullet$ to denote a variable input to a functor, e.g. $\Hom_R(\bullet,M)$ for the contravariant Hom functor and module $M$. In this section, we briefly review the theory of modules with Frobenius actions. Throughout, $(R,\mfm)$ denotes a local ring of characteristic $p>0$. The following definition provides one of our primary tools for the results that follow. 

\begin{defn}\label{defn: R[F]-module basics}
Let $M$ and $N$ be $R$-modules. An additive map $\rho:M\rightarrow N$ is called a \textbf{$p^e$-linear map} for some $e\in\mbn$ if $\rho(rm)=r^{p^e}\rho(m)$ for all $r\in R$ and $m\in M$. A $p$-linear endomorphism of $M$ is called a \textbf{Frobenius action} on $M$. Given $\rho:M\rightarrow M$ a Frobenius action and $N\subset M$ a submodule, we say $N$ is \textbf{$\rho$-stable} if $\rho(N)\subset N$. Finally, if $\rho: M\rightarrow M$ and $\rho' \colon M'\rightarrow M'$ are Frobenius actions, an $R$-linear map $\phi \colon M\rightarrow M'$ is said to \textbf{commute with Frobenius} if $\phi\circ \rho = \rho'\circ\phi$.
\end{defn}

The most commonly studied examples of modules with Frobenius actions are Frobenius actions induced by the Frobenius map $F \colon R\rightarrow R$, especially the canonical action on the local cohomology modules of $R$, see \Cref{xmp: lc has frob action}. It is also useful to interpret modules with Frobenius actions from a categorical perspective.

\begin{rmk}\label{rmk: Frobenius actions and R[F]-modules}
Consider the category whose objects are pairs $(M,\rho)$ where $M$ is an $R$-module and $\rho:M\rightarrow M$ is a Frobenius action, and whose morphisms $\phi \colon (M,\rho)\rightarrow (M',\rho')$ are $R$-linear maps which commute with Frobenius. This category is well-known to be equivalent to the category of left modules over the non-commutative ring $R[F]$, formed by adjoining a non-commuting variable $\chi$ to $R$ subject to the relations $r^{p}\chi = \chi r$ for each $r \in R$. See, for example, \cite[Def.~2.3]{EH08} for an explanation of this perspective. 

Due to this equivalence of categories, we say that $(M,\rho)$ is an $R[F]$-module to mean that $M$ is an $R$-module and $\rho$ is a Frobenius action on $M$. Moreover, $(N,\rho|_{N})$ is an $R[F]$-submodule of $(M,\rho)$ if $N$ is a $\rho$-stable $R$-submodule of $M$. If we say $(M,\rho)$ is a finitely generated or artinian $R[F]$-module, we mean that $M$ is finitely generated or artinian over $R$. Finally, we will repeatedly use the fact that a category of left modules over a non-commutative ring forms an abelian category without any further comment.

One final perspective we may use to discuss Frobenius actions is by twisting to get an $R$-linear map. In particular, let $F_*(\bullet)$ be the functor on the category of $R$-modules obtained by restricting scalars along the Frobenius map $F:R\rightarrow R$. One may easily check that a Frobenius action $\rho: M\rightarrow M$ is equivalent to an $R$-linear map $\tilde{\rho}: M\rightarrow F_*(M)$.
\end{rmk}

Since the category of $R[F]$-modules is an abelian category, we get natural $R[F]$-module structures on quotients and direct sums of $R[F]$-modules. We define the former below.

\begin{defn}\label{dff: quotient R[F]-structure}
Given an $R[F]$-module $(M,\rho)$ and a $R[F]$-submodule $N\subset M$, the $R$-module $M/N$ has a canonical Frobenius action $(M/N,\overline{\rho})$ given by $\overline{\rho}(m+N)=\rho(m)+N$. 
\end{defn}

Of interest when studying $R[F]$-modules, is the kernel of all iterates of the Frobenius action, which is a special case of the definition below.

\begin{defn}
Given an $R[F]$-module $(M,\rho)$ and a $\rho$-stable submodule $N\subset M$, the \textbf{$\rho$-orbit closure of $N$ in $M$} is the $R$-submodule given by \[N^\rho_M := \{ m \in M \mid \rho^e(m)\in N \text{ for some } e \in \mbn\}.\] The submodule $N$ is said to be \textbf{$\rho$-nilpotent} (or just \textbf{nilpotent}) if $N\subset 0^\rho_M$ and  \textbf{$\rho$-closed} if $N=N^\rho_M$. Finally, we will say the $R[F]$-module $M$ is \textbf{generalized nilpotent} if $M/0^\rho_M$ is a finite length $R$-module. 
\end{defn}

It is clear that if $N\subset M$ is an $R[F]$-submodule, then $\pi \colon M\rightarrow M/N$ commutes with Frobenius and $\pi^{-1}(0^{\overline{\rho}}_{M/N}) = N^\rho_M$. The following lemma helps us understand how nilpotence is tracked along short exact sequences; see \cite[Lem.~2.11]{MM} for a proof.

\begin{lemma}\label{lem: nilpotent direct sum}
If $0\to A \to C \to B \to 0$ is a short exact sequence of $R[F]$-modules, then $C$ is nilpotent if and only if $A$ and $B$ are nilpotent. In particular, $A\oplus B$ is nilpotent if and only if $A$ and $B$ are nilpotent.
\end{lemma}

In the study of nilpotent Frobenius actions, it is often important to keep track of the smallest iterate $e$ for which $\rho^e=0$. 

\begin{defn}
Given an $R[F]$-module $(M,\rho)$, the \textbf{Hartshorne-Speiser-Lyubeznik number of $(M,\rho)$} is $\hsl(M,\rho) :=\inf \{e \in \NN \mid \rho^e(0^\rho_M)=0\}\in \NN\cup\{\infty\}.$ 
\end{defn}

\begin{rmk}\label{rmk: HSL numbers are finite}
If $(M,\rho)$ is a finitely generated $R[F]$-module, then clearly $\hsl(M)<\infty$. In the case that $M$ is artinian, Hartshorne-Speiser \cite[Prop.~1.11]{HS77}, Lyubeznik \cite[Prop.~4.4]{Lyu97}, and Sharp \cite[Cor.~1.8]{Sha06} proved that $\hsl(M)<\infty$.
\end{rmk}

We conclude this subsection by describing the most important example of an $R[F]$-module for our purposes, that is, how a Frobenius action on a module $M$ induces a Frobenius action on the local cohomology modules of $M$. For another treatment (in the case $M=R$), we refer the reader to \cite[Rem.~2.1]{Sha06}.

\begin{example}\label{xmp: lc has frob action}
Let $(M,\rho)$ be a finitely generated $R[F]$-module and let $I\subset R$ be an ideal. We view $\rho$ as a $R$-linear map $\tilde{\rho} \colon M\rightarrow F_*(M)$. We can apply the local cohomology functor $H^j_I(\bullet)$ to $\tilde{\rho}$ to get an $R$-linear map which we also call $\tilde{\rho}$ by minor abuse of notation \[
\tilde{\rho}\colon H^j_I(M)\rightarrow H^j_I(F_*(M)).
\] 

Since $F_*$ is a restriction of scalars functor, which is exact, we have $H^j_I(F_*(M))$ is isomorphic to $F_*(H^j_I(M))$, so $\tilde{\rho}:H^j_I(M)\rightarrow F_*(H^j_I(M))$ is equivalent to a Frobenius action on $H^j_I(M)$.
\end{example}

\subsection{Graded rings and Frobenius}

We now discuss some preliminaries involved in studying graded rings and the Frobenius map.

\begin{rmk}\label{rmk: graded setting}
We will say that $R$ is a graded ring to mean that $R$ is an $\mbn$-graded ring where $[R]_0$ is a local ring $(S,\mfn)$ of prime characteristic $p>0$, and write $\mfm = \mfn + R_+$. For our applications, it is safe to assume that $R$ is generated over $S$ by elements of degree $1$, i.e., $R=S[R_1]$, but this may not be necessary in the results that follow. By a graded $R$-module, we mean a $\ZZ$-graded $R$-module. For a graded $R$-module $M$ and an $n\in\mbz$, we let $M(n)$ denote the graded $R$-module with grading given by $[M(n)]_m=[M]_{n+m}$. We will often refer to $M(n)$ as a \textbf{graded twist} of $M$.
\end{rmk}

As we need to use Frobenius actions in the context of Rees algebras and associated graded rings, we need to study the appropriate notion of graded Frobenius action.

\begin{defn}
If $M$ is a graded $R$-module, a $p$-linear map $\rho:M\rightarrow M$ is a \textbf{graded Frobenius action} if $\rho([M]_n)\subset [M]_{np}$ for all $n \in \mbz$. We will call $(M,\rho)$ a \textbf{graded $R[F]$-module} if $\rho:M\rightarrow M$ is a graded Frobenius action.
\end{defn}

Since a graded Frobenius action multiplies the degree of a homogeneous element by $p$, if $(M,\rho)$ is a graded $R[F]$-module then $0^\rho_M$ is a homogeneous $\rho$-stable submodule of $M$. We will also need to understand how Frobenius actions are affected by graded twists.

\begin{defn} \label{dff: twisted action}
Let $(M,\rho)$ be a graded $R[F]$-module and let $n\in \ZZ$. The graded twist $M(n)$ does not carry a natural Frobenius action, but $\rho \colon M\rightarrow M$ induces a natural $p$-linear map $^{(n)}\rho \colon M(n)\rightarrow M(np)$ which we will refer to as a \textbf{twisted action}. 
\end{defn}

Since $H^j_\mfm(M(n)) = H^j_\mfm(M)(n)$, a twisted action on $M$ (as in the previous definition) will induce a twisted action on the local cohomology modules of $M$ via a simple graded adjustment to \Cref{xmp: lc has frob action}.

\begin{rmk} \label{rmk: twisted action nilpotent}
Let $M$ be a graded $R[F]$-module and let $n\in \ZZ$. Then, $\rho^e\colon M\rightarrow M$ vanishes if and only if $^{(n)}\rho^e \colon M(n)\rightarrow M(np^e)$ vanishes, since $^{(n)}\rho$ and $\rho$ agree after applying the forgetful functor from the category of graded $R[F]$-modules to the category of $R[F]$-modules. 
\end{rmk}

In the ungraded case, if $(M,\rho)$ is an $R[F]$-module, then for any $x \in R$ we have that $x\rho \colon M\rightarrow M$ is another Frobenius action on $M$. When $R$ and $M$ are graded and $x$ is homogeneous, we can incorporate a twist by $\deg x$ to obtain another Frobenius action.

\begin{rmk}\label{rmk: adjusted Frob actions}
Let $(M,\rho)$ be a graded $R[F]$-module and let $x \in R$ be homogeneous with $\deg x = m$. Then, $x^{p-1}\rho \colon M(-m)\rightarrow M(-m)$ is a graded Frobenius action on $M(-m)$, as if $\xi \in [M(-m)]_n$ then $x^{p-1}\rho(\xi) \in [M(-m)]_{np}$.  Further, the map $x^{p-1}\cdot \colon M\rightarrow M$ given by $x^{p-1}\cdot(\xi)=x^{p-1}\xi$ induces the same map $x^{p-1}\cdot\colon H^j_\mfm(M)\rightarrow H^j_\mfm(M)$, so we must have that the action $x^{p-1}\rho\colon M(-m)\rightarrow M(-m)$ induces the Frobenius action $x^{p-1}\rho\colon H^j_\mfm(M)(-m) \rightarrow H^j_\mfm(M)(-m)$.
\end{rmk}

The following definition is useful when studying Frobenius actions on quotient rings and works in either the graded or ungraded setting.

\begin{defn}\label{dff: relative action}
Let $I\subset R$ be a (homogeneous) ideal. Recall the \textbf{Frobenius bracket powers} of $I$, i.e. if $I=(x_1,\cdots,x_t)$, then $I^{\fbp{p^e}}$ is the ideal $I^{\fbp{p^e}} = (x_1^{p^e},\cdots, x_t^{p^e})$. We have the following commutative diagram of (graded) $R$-modules. 
\begin{center}
    \begin{tikzcd}
    R/I\arrow{dr}[swap]{f^e_R} \arrow{rr}{F^e} && R/I \\
    & R/I^{\fbp{p^e}} \arrow{ur}[swap]{\pi_e} &
    \end{tikzcd}
\end{center} 
where $F^e$ is the Frobenius map on $R/I$, $\pi_e$ is the canonical projection map, and $f^e_R$ is the $p^e$-linear map defined by $f^e_R(x+I) = x^{p^e}+I^{\fbp{p^e}}$. The $p$-linear map $f_R \colon R/I\rightarrow R/I^{\fbp{p}}$ is sometimes referred to in the literature as the \textbf{relative Frobenius action}.
\end{defn}

The following result about the Frobenius map on a quotient ring is well-known and used implicitly and explicitly in the literature surrounding quotients of rings of prime characteristic. We state it for graded rings, but it also trivially applies to a local ring. 

\begin{lemma}\label{lem: graded seses associated to homogeneous elements}
Let $(R,\mfm)$ be a graded ring and let $x\in R$ be a homogeneous regular element with $\deg(x)=m$. Then for any $e \in \mbn$, we have the following two commutative diagrams with exact rows, whose horizontal maps are homogeneous $R$-linear maps of degree $0$ and whose vertical maps are graded $p^e$-linear over $R$.
\begin{center} 
\begin{tikzcd}
    0 \arrow{r} & R(-m)\arrow{r}{\cdot x}\arrow{d}{^{(-m)}F^e} & R\arrow{d}{F^e} \arrow{r} & R/xR\arrow{d}{f^e_R} \arrow{r} & 0 \\
    0 \arrow{r} & R(-mp^e) \arrow{r}{\cdot x^{p^e}} & R \arrow{r} & R/x^{p^e}R \arrow{r} & 0 \\ 
    &&&& \\[-2em]
    0 \arrow{r} & R(-m)\arrow{r}{\cdot x}\arrow{d}{x^{p^e-1}F^e} & R\arrow{d}{F^e} \arrow{r} & R/xR\arrow{d}{F^e} \arrow{r} & 0 \\
    0 \arrow{r} & R(-m) \arrow{r}{\cdot x} & R \arrow{r} & R/xR \arrow{r} & 0
    \end{tikzcd} 
\end{center}
\end{lemma}
\begin{proof}
Recall $S=[R]_0$. For $n \in \mbn$, we get a commutative diagram with exact rows whose horizontal maps are $S$-linear, and whose vertical maps are $p^e$-linear over $S$ between the first two rows and $S$-linear between the last two rows.
\begin{center} \begin{tikzcd}
0 \arrow{r} & \left[R\right]_{n-m}\arrow{d}{^{(-m)}F^e} \arrow{r}{\cdot x} & \left[R\right]_n\arrow{d}{F^e} \arrow{r} & \left[R/xR\right]_n \arrow{d}{f^e_R} \arrow{r} & 0 \\
0 \arrow{r} & \left[R\right]_{np^e-mp^e}\arrow{d}{ \cdot x^{p^e-1}} \arrow{r}{\cdot x^{p^e}} & \left[R\right]_{np^e}\arrow[equals]{d} \arrow{r} & \left[R/x^{p^e}R\right]_{np^e}\arrow{d}{\pi} \arrow{r} & 0 \\
0 \arrow{r}& \left[R\right]_{np^e-m} \arrow{r}{\cdot x} & \left[R\right]_{np^e} \arrow{r} & \left[R/xR\right]_{np^e} \arrow{r} & 0
\end{tikzcd} \end{center} The conclusion then follows from \Cref{rmk: twisted action nilpotent} and \Cref{dff: twisted action}. 
\end{proof}

We next give generalizations of two important results used in studying Frobenius actions on local cohomology. The first utilizes filter regular sequences; as for the ungraded case, a sequence $\uline{x}=x_1,\cdots,x_t$ of homogeneous elements is \emph{filter regular} if for each $0\le i < t$, the image of $(x_1,\cdots,x_i):x_{i+1}$ forms a finite length ideal in $R/(x_1,\cdots,x_i).$ Filter regular sequences are crucial to the Nagel-Schenzel isomorphism, see \cite[Lem.~3.4]{NS94}; we now generalize this result to the graded $R[F]$-module case, which may be of independent interest.

\begin{thm}\label{thm: graded NS}
Suppose $R$ is a graded ring with homogeneous maximal ideal $\mfm$. Let $M$ be a finitely generated graded $R[F]$-module and suppose $\uline{x}=x_1,\cdots,x_t$ is a homogeneous filter regular sequence on $M$. 
For each $1 \leq j < t$, we have an isomorphisms of graded $R[F]$-modules \[
H^j_\fm(M) \cong H^{j-t}_\fm(H^t_{(\ux)}(M))
\] for $j\geq t$ and $H^j_\fm(M)\cong H^j_{(\ux)}(M)$ for $j<t$.
\end{thm}

\begin{proof}
The proof follows completely the strategy of \cite[Thm~1.1]{Huo17}, with minor adaptations to keep track of the grading and the Frobenius action. Set $$C^{\sbt} :=
\begin{tikzcd}
    0 \arrow{r} & M \arrow{r}{d^0} & \oplus M_{x_i} \arrow{r}{d^1} & \cdots \arrow{r}{d^{t-1}} & M_{x_1\cdots x_t} \arrow{r} & 0
\end{tikzcd}$$ the \v{C}ech complex, whose $j$-th cohomology is isomorphic to $H_{(\ux)}^j(M)$. The Frobenius action on $M$ induces the $R[F]$-module structure on $H_{(\ux)}^j(M)$ as in \Cref{xmp: lc has frob action}. Additionally, the coboundary maps in the \v{C}ech complex are homogeneous and $R[F]$-linear, so the \v{C}ech complex is a complex of graded $R[F]$-modules. Further, the category of graded $R[F]$-modules is abelian, whence kernels and cokernels remain graded $R[F]$-modules. 

As in the proof of \cite[Thm~1.1]{Huo17}, one splits the \v{C}ech complex into short exact sequences as follows.  For positive integer $i$, set $L^i = \im d^{i-1}$ and $K^i = \ker d^i$ and note one has short exact sequences \[\begin{tikzcd}[row sep=2pt]
    0 \arrow{r} & L^i \arrow{r} & K^i \arrow{r} & H_{(\ux)}^i(M) \arrow{r} & 0, & \\
    0 \arrow{r}& K^i \arrow{r} & C^i \arrow{r} & L^{i+1} \arrow{r} & 0 & \text{ for } i>0, \text{ and} \\
    0 \arrow{r} & H^0_{(\ux)}(M) \arrow{r} & M \arrow{r} & L^1 \arrow{r} & 0. & 
\end{tikzcd}\]
Using the long exact sequences produced by applying $H_\fm^j(\Arg)$ to these sequences, one next notes the following claims, compare to \cite[Lem. 2.3, 2.5, Cor. 2.4]{Huo17}.

\medskip

\noindent {\bf Claim:} The following hold.
\begin{enumerate}
\item Each $C^i$ satisfies $H_\fm^j(C^i) \cong 0$ for all $j \geq 0$, in particular $H_\fm^0(L^i) \cong H_\fm^0(K^i) \cong 0$ for $i \geq 1$, by left exactness.
\item For $j < t$, $H_{(\ux)}^j(M)$ is $\fm$-torsion.
\end{enumerate}

By induction on $j$, it now follows that for $j < t$, $H_\fm^j(L^1) \cong  H_\fm^1(L^j)$ whence $$H_\fm^j(M) \cong H_\fm^j(L^1) \cong  H_\fm^1(L^j) \cong H_{(\ux)}^j(M).$$ Similarly, for $j \geq t$, one obtains that $$H_\fm^j(M) \cong H_\fm^j(L^1) \cong H_\fm^{j-t+1}(L^t) \cong H_\fm^{j-t}( H_{(\ux)}^t(M)),$$ as desired.

Hence, it suffices to ensure the claims hold in the category of graded $R[F]$-modules. The first claim is clear by definition. For the second claim, it is also clear by definition that for a homogeneous filter regular sequence $\ux$, the module $((x_1,\ldots,x_i)M :_M x_i)/(x_1,\ldots,x_i)M$ is annihilated by a power of $\fm$, whence $\ux$ becomes a possibly improper homogeneous regular sequence on the homogeneous localization $M_\fm$. However, then $H_{(\ux) R_\fm}^j(M_\fp) \cong 0$ for $j < t$ for all $\fp$ in the punctured spectrum whence $H_{(\ux)}^j(M)$ is $\fm$-torsion for $i < t$. 
\end{proof}

\begin{rmk}
By endowing the module $M$ in the previous theorem with the trivial Frobenius action sending every element to 0, the above provides a proof of a homogeneous version of \cite[Lem.~3.4]{NS94}. Furthermore, minor adjustments to the proof above also provide a local $R[F]$-module version of \cite[Lem.~3.4]{NS94}. 
\end{rmk}

\begin{rmk}\label{rmk: nagel-schenzel nilpotence}
The preceding theorem allows us to understand Frobenius actions on lower local cohomology modules by using a top local cohomology module. In particular, suppose $R$, $\uline{x}$, and $M$ are as in the statement of \Cref{thm: graded NS}. Write $\rho$ for the Frobenius action on $M$, and as before, continue to use $\rho$ for the Frobenius action induced on the graded local cohomology modules of $M$. If $\xi \in H^j_\mfm(M)$ is a homogeneous element, we can view $\xi$ as an $\mfm$-torsion cohomology class $[m+(x_1^n,\cdots,x_j^n)M] \in H^j_{(x_1,\cdots,x_j)}(M)$, whose Frobenius action is given by \[\rho([m+(x_1^n,\cdots,x_j^n)M]) = [\rho(m)+(x_1^{np},\cdots,x_j^{np})M].\] Furthermore, $\xi$ is nilpotent in $H^j_\mfm(M)$ if and only if, for some $e \gg 0$, we have \[ \left[\rho^e(m)+(x_1^{np^e},\cdots,x_j^{np^e})M\right]=0\] in $H^j_{(x_1,\cdots,x_j)}(M)$.
\end{rmk}

Utilizing the Nagel-Schenzel isomorphism, in the local setting Polstra-Quy \cite{PQ19} studied how $F$-depth is related to the relative Frobenius action on quotients by a filter regular element. By using the graded Nagel-Schenzel isomorphism above and \Cref{lem: graded seses associated to homogeneous elements}, we obtain the following graded version of their result \cite[Thm.~4.2]{PQ19}. 

\begin{thm}\label{thm: graded PQ 4.2}
Let $R$ be a graded ring of dimension $d>0$ as in \Cref{rmk: graded setting}. For every $t>0$, the following are equivalent.
\begin{enumerate}[label=(\alph*)]
    \item The $F$-depth of $R$ satisfies $\fdp (R) > t$.
    \item For every homogeneous filter regular element $x \in R$ and $j<t$, each element of $H^j_\mfm(R/xR)$ vanishes under $f^e_R$ for some $e \in \NN$.
    \item There is a homogeneous filter regular element $x\in R$ such that for all $n \ge 1$ and $j<t$, each element of $H^j_\mfm(R/x^nR)$ vanishes under $f^e_R$ for some $e \in \NN$.
    \item There is a homogeneous filter regular element $x \in R$ such that for all $e' \ge 0$ and $j<t$, each element of $H^j_\mfm(R/x^{p^{e'}}R)$ vanishes under $f^e_R$ for some $e \in \NN$.
\end{enumerate}
\end{thm}

The proof of this theorem follows by performing an identical diagram chase as in the proof of \cite[Thm.~4.2]{PQ19} after applying the graded local cohomology functor to the commutative diagrams in \Cref{lem: graded seses associated to homogeneous elements}.

\section{Generalized F-depth with respect to an ideal}\label{sec: Fdepth}

Throughout this section, let $(R,\mfm)$ be a local ring of prime characteristic $p>0$ and of dimension $d$ and let $J\subset R$ be an ideal. We begin by defining a depth-like invariant, which can be seen as a Frobenius version of the generalized depth, $\gdp_J(R)$, of Huckaba-Marley \cite{HM94}, as a generalization of Lyubeznik's $F$-depth, $\fdp(R)$, \cite[Def.~4.1]{Lyu06}, and the generalized $F$-depth of Maddox-Miller \cite[Def.~3.3]{MM}. For convenience of the reader, we recall that 
\[
\gdp_J(R) = \inf\{ j \in \NN \mid J^N H^j_\mfm(R)\neq 0 \text{ for any } N\in \NN\},
\] and  \[
\fdp(R) = \inf\{j \in \NN \mid H^j_\mfm(R) \text{ is not nilpotent under } F\}.
\] For the definition of the generalized $F$-depth, take $J=\mfm$ in the definition below.

\begin{defn}
\label{defn: gFdepth}
Let $(M,\rho)$ be an $R[F]$-module and let $J\subset R$ be any ideal. The \textbf{generalized $F$-depth of $(M,\rho)$ with respect to $J$} is \[
\gfdp_J(M) := \inf \{j \in \mbn \mid (J^N H^j_\mfm(M),\rho) \text{ is not nilpotent for any } N\in \mbn \} \in \mbn \cup\{\infty\}.
\]
\end{defn}

To better understand generalized $F$-depth with respect to an ideal, it will be useful to explicate equivalent conditions for when $(J^NH^j_\mfm(M),\rho)$ is nilpotent.  

\begin{lemma} \label{lem: equivalent conditions for J^NM to be nilpotent}
For any $R[F]$-module $(M,\rho)$ with finite Hartshorne-Speiser-Lyubenik number $e_0$, the following are equivalent.
\begin{enumerate}[label=(\alph*)]
\item The $R[F]$-module $(J^N M,\rho)$ is nilpotent for some $N \ge 0$.
\item There is an $N\ge 0$ such that for all $x \in J^N$ and any $e \ge e_0$, the $p^e$-linear map $x^{p^e}\rho^e\colon M \rightarrow M$ vanishes. 
\item The ideal $J$ annihilates $M/0^\rho_M$ up to radical.
\end{enumerate}
\end{lemma}

\begin{proof}
For any $x \in R$, the $R[F]$-module $(xM,\rho)$ is nilpotent if and only if, for all $m \in M$ 
\[
\rho^e(xm)=x^{p^e}\rho^e(m)=0
\] 
or, equivalently, if and only if $xm \in \ker(\rho^e)=0^\rho_M$ as $e\ge e_0$. Thus, $(xM,\rho)$ is nilpotent if and only if $x \in \ann_R(M/0^\rho_M)$. Then, $(J^NM,\rho)$ is nilpotent if and only if for all $x \in J^N$, $(xM,\rho)$ is nilpotent, which is equivalent to both $x^{p^e}\rho^e \colon M \rightarrow M$ being the zero map and $x \in \ann_R(M/0^\rho_M)$, which gives $J^N\subset \ann_R(M/0^\rho_M)$.
\end{proof}

We are mostly interested in studying $\gfdp_J(M)$ in the case that $(M,\rho)$ is either $R$ itself, a quotient or localization of $R$, or an ideal in the ring $R$, each of which we endow with an $R[F]$-module structure via the canonical Frobenius map $F$ (or its restriction). In these cases, the local cohomology modules of $M$ have finite Hartshorne-Speiser-Lyubeznik numbers (see \Cref{rmk: HSL numbers are finite}) and so the result above applies to the $R[F]$-module $(H^j_\mfm(M),\rho)$.

\begin{rmk}\label{rmk: gF-depth_J unifies F-depth and gF-depth}
The notion of generalized $F$-depth with respect to an ideal recovers both the $F$-depth of Lyubeznik \cite[Def.~4.1]{Lyu06} and the generalized $F$-depth of Maddox-Miller \cite[Def.~3.3]{MM} as special cases. In particular, if we let $J=R$, then $(R^NM,\rho)=(M,\rho)$ is nilpotent for some $N\ge 0$ if and only if $(M,\rho)$ is nilpotent; if $J=\mfm$, then $(\mfm^N M,\rho)$ is nilpotent for some $N\ge 0$ if and only if $(M,\rho)$ is generalized nilpotent.

Thus, the $F$-depth of an $R[F]$-module $(M,\rho)$ is the same as its generalized $F$-depth with respect to the unit ideal, and similarly its generalized $F$-depth coincides with its generalized $F$-depth with respect to the maximal ideal $\mfm\subset R$. Consequently, any result about generalized $F$-depth with respect to an arbitrary ideal $J$ also holds for $F$-depth and generalized $F$-depth. 
\end{rmk}

As for depth and $F$-depth, we will be interested in when this new invariant is maximal.

\begin{defn}\label{dff: gwFn wrt J}
Let $J\subset R$ be an ideal. We will say that $R$ is \textbf{generalized weakly $F$-nilpotent with respect to $J$} if $\gfdp_J (R)=\dim (R)$.
\end{defn}

As the local theory of depth is closely related to the annihilators of local cohomology modules, the local theory of (generalized) $F$-depth is tied to the colon ideals $0^F_{H^j_\mfm(M)}:_R H^j_\mfm(M)$. 

\begin{lemma}\label{lem: gF-depth_J observations} 
Let $J\subset R$ be an ideal, $(M,\rho)$ be a finitely generated $R[F]$-module with $\dim_R(M)=t$, and for any $0 \le j \le t$, set \[\mfb_j(M) = \sqrt{\ann_R\left(H^j_\mfm(M)/0^F_{H^j_\mfm(M)}\right)} \text{ and } \mfb(M) = \mfb_0(M)\cap \cdots\cap \mfb_{t-1}(M).\]
\begin{enumerate}[label=(\alph*)]
\item We have $\gfdp_J(M) < \infty$ if and only if $J\not\subset \mfb(M)\cap \mfb_t(M)$. In particular, if $J\subset\sqrt{0}$ then $\gfdp_J(M)=\infty$. 
\item If $I$ is an ideal contained in $J$, then $\gfdp_J(M) \le \gfdp_I(M)$, and equality holds if $\sqrt{I}=\sqrt{J}$. 
In particular, we have the following inequalities for any ideal $J\subset \mfm$. 
\[
\depth(M) \le \fdp(M) \le \gfdp(M) \le \gfdp_J(M)
\] 
\item If $J\neq R$, then $\gfdp_J(M) >0$. 
\item For any $k \ge 1$, if $J \subset \mfb_j(M)$ for $0 \le j < k$, then $\gfdp_J(M)\ge k$, in particular, $\gfdp_J(M) =k$ if and only if $J\subset \mfb_0(M)\cap \cdots\cap \mfb_{k-1}(M)$ and $J\not\subset \mfb_k(M)$. Thus, $R$ is generalized weakly $F$-nilpotent with respect to $J$ if and only if $J\subset \mfb(R)$. 
\item If $(R,\mfm)\rightarrow (S,\mfn)$ is a local homomorphism with $\sqrt{\mfm S}=\mfn$, then for any $S[F]$-module $(N,\rho')$ and any ideal $J\subset R$, we have $\gfdp_J (N)=\gfdp_{JS} (N)$ where we view $N$ as an $R[F]$-module via restriction of scalars.
\end{enumerate}
\end{lemma}

\begin{proof}
Claims (a) and (d) follow immediately from \Cref{lem: equivalent conditions for J^NM to be nilpotent}. Claim (b) is obvious, since each of the quantities involved measures a more generalized version of vanishing than the previous. Claim (c) is immediate from the definition of $\gfdp_J(R)$, as $H^0_\mfm(M)$ is finitely generated and $\mfm$-torsion. 

To see claim (e), notice that the change-of-rings isomorphism \cite[Prop. 7.15(2)]{ILL07} for local cohomology also holds in the category of $R[F]$-modules since local cohomology is the cohomology of the \v{C}ech complex and formation of this complex is compatible with restriction of scalars, which gives that $H^j_\mfm(N)$ is isomorphic to $H^j_{\mfm S}(N)$ as $R[F]$-modules. Further, since $\sqrt{\mfm S}=\mfn$, we know $H^j_{\mfm S}(N)$ is isomorphic to $H^j_{\mfn}(N)$ as $R$-modules. To show this isomorphism commutes with the induced Frobenius actions in each place, one uses the definition of local cohomology as the direct limit of Ext modules and functoriality of the Hom functor in the second factor with respect to Frobenius actions. Finally, since $R$ acts as its image under $S$, we see that $J^n H^j_\mfm(N)\simeq (J^nS)H^j_\mfn(N)$, which concludes the proof.
\end{proof}

We now demonstrate that generalized $F$-depth with respect to $J$ exhibits one of the most important and useful properties of depth, namely it satisfies expected inequalities over short exact sequences. Note that this is a strict improvement to \cite[Lem.~2.23]{MM}.

\begin{thm}\label{thm: F-depth lemma}
Let $J\subset R$ be an ideal, and suppose that we have a short exact sequence of $R[F]$-modules as below, where each term in the sequence has finite Hartshorne-Speiser-Lyubeznik number. \begin{center}
\begin{tikzcd}
0 \arrow{r} & (A,\rho_A) \arrow{r}& (B,\rho_B) \arrow{r} & (C,\rho_C) \arrow{r} & 0
\end{tikzcd}
\end{center} Then, we have the following inequalities, where we interpret $\infty\pm 1$ as $\infty$.

\begin{enumerate}
\item $\gfdp_J(B) \ge \min \{\gfdp_J(A),\gfdp_J(C)\}$
\item $\gfdp_J(A) \ge \min \{\gfdp_J(B),\gfdp_J(C)+1\}$
\item $\gfdp_J(C)\ge \min \{\gfdp_J(B),\gfdp_J(A)-1\}$
\end{enumerate}
Furthermore, if the sequence is split exact, then equality in $(1)$ is attained.
\end{thm}

\begin{proof}
We will prove $(1)$, as the proofs of the other inequalities are similar. We first assume that $\gfdp_J(A)$ and $\gfdp_J(C)$ are finite. Let $t= \min\{\gfdp_J(A),\gfdp_J(C)\}$. Thus, for any $N$ and any $x \in J^N$, $0 \le j < t$, and $e \ge \max\{\hsl(H^j_\mfm(A)),\hsl(H^j_\mfm(C))\}$, we get the following commutative diagram of $R$-modules, where the horizontal maps are $R$-linear and the vertical maps are $p^e$-linear

\begin{center}
    \begin{tikzcd}
        \cdots \arrow{r} & H^j_\mfm(A) \arrow{d}{x^{p^e} \rho^e_A} \arrow{r}{\alpha} & H^j_\mfm(B) \arrow{r}{\beta}  \arrow{d}{x^{p^e} \rho^e_B} & H^j_\mfm(C)\arrow{r}  \arrow{d}{x^{p^e} \rho^e_C} & H^{j+1}_\mfm(A)\arrow{r} \arrow{d}{x^{p^e}\rho^e_A} & \cdots \\
    \cdots \arrow{r} & H^j_\mfm(A) \arrow{r}{\alpha} & H^j_\mfm(B) \arrow{r}{\beta} & H^j_\mfm(C)\arrow{r} & H^{j+1}_\mfm(A)\arrow{r} & \cdots .
    \end{tikzcd}
\end{center}

Then, $x^{p^e}\rho^e_C = 0$, so $\beta \circ x^{p^e}\rho^e_B = 0$. Hence, the image of $x^{p^e}\rho^e_B$ is in the image of $\alpha$, but the image of $\alpha$ vanishes under $x^{p^e}\rho^e_B$ since $x^{p^e}\rho^e_A=0$. This shows $x^{p^{2e}} \rho^{2e}_B$ vanishes, and as $x\in J^N$ was arbitrary, we may apply \Cref{lem: equivalent conditions for J^NM to be nilpotent} to conclude $J^{N'} H^j_\mfm(B)$ is nilpotent for some $N' \gg 0$. As this holds for all $0 \le j < k$, we have the required inequality.

If at least one of $\gfdp_J(A)$ or $\gfdp_J(C)$ is infinite, say $\gfdp_J(A)$, then $x^{p^e}\rho_A^e:H^j_\mfm(A)\rightarrow H^j_\mfm(A)$ vanishes for all $j \in \mbn$. From the commutative diagram above, it is clear that if $e\gg 0$, the vanishing of $x^{p^e}\rho^e_B$ is equivalent to the vanishing of $x^{p^e}\rho^e_C$. This justifies the infinite arithmetic.

Finally, in the case that the sequence is split, by \Cref{lem: nilpotent direct sum} we get that $J^NH^j_\mfm(A\oplus C) = J^NH^j_\mfm(A) \oplus J^NH^j_\mfm(C)$ is nilpotent if and only if $J^NH^j_\mfm(A)$ and $J^NH^j_\mfm(C)$ are nilpotent, which demonstrates the required equality.
\end{proof}

For any $R[F]$-module $(M,\rho)$ and any ideal $I\subset R$, the submodule $IM\subset M$ is  $\rho$-stable, since $\rho(IM)\subset I^{\fbp{p}} \rho(M) \subset IM$. We now demonstrate that we can check its generalized $F$-depth with respect to $J$ up to radical.

\begin{lemma}\label{lem: checking gfdp_J up to radical}
Let $(M,\rho)$ be an $R[F]$-module and let $N\subset M$ be $\rho$-stable, and endow $M/N$ with the natural action $\overline{\rho}$ given in \Cref{dff: quotient R[F]-structure}. We then have \[ \gfdp_J(M/N,\bar{\rho}) = \gfdp_J(M/N_M^\rho,\bar{\rho}).\] Hence, for any ideal $I \subset R$, we have $\gfdp_J (M/IM) = \gfdp_J (M/\sqrt{I}M) $.
\end{lemma}

\begin{proof}
Since $N^\rho_M/N$ is nilpotent in $M/N$, we have $\fdp (N^\rho_M/N) = \infty$, which gives the first equality. To prove the second, notice that $IM\subset\sqrt{I}M\subset (IM)^\rho_M$, as for any $x \in \sqrt{I}$ and $m \in M$, $x^{p^e}\in I$ for some $e\in \NN$, whence $\rho^e(xm) = x^{p^e}\rho^e(m)\in IM$. By applying $\bullet^\rho_M$ to the chain of containments above, we then deduce $(IM)^\rho_M =(\sqrt{I}M)^\rho_M$ . 
\end{proof}

As for depth, the generalized $F$-depth with respect to an ideal is unaffected by completion, and is even better behaved with respect to going modulo the nilradical than ordinary depth. This was shown by Lyubeznik for $F$-depth, see \cite[Props.~4.4 and 4.5]{Lyu06}.

\begin{thm}\label{thm: gfdp and reduction + flat local extension}
Let $J\subset R$ be an ideal, and let $\widehat{R}$ be the completion of $R$ at its maximal ideal and $R_{\text{red}}=R/\sqrt{0}$ be its reduction. \begin{enumerate}
 \item Unlike ordinary depth, $\gfdp_J(R)=\gfdp_{JR_{\text{red}}}(R_{\text{red}})$. 
\item If $\phi \colon (R,\mfm) \rightarrow (S,\mfn)$ is a flat local extension with $\sqrt{\mfm S}=\mfn$, then we have $$\gfdp_J(R)=\gfdp_{JS}(S).$$ In particular, $\gfdp_J(R) = \gfdp_{J\widehat{R}}(\widehat{R})$. 
\end{enumerate}
\end{thm}

\begin{proof}
For $(1)$, we first see that $\gfdp_J (R) = \gfdp_{JR_{\text{red}}} (R_{\text{red}})$ from \Cref{lem: checking gfdp_J up to radical,lem: gF-depth_J observations}.

For $(2)$, since $S$ is faithfully flat over $R$, the $p^e$-linear map $x^{p^e} F^e:H^j_\mfm(R)\rightarrow H^j_\mfm(R)$ vanishes if and only if it vanishes after applying $\bullet \otimes_R S$, from which we obtain the $p^e$-linear map $\phi(x)^{p^e} F^e \colon H^j_{\mfm S}(S) \rightarrow H^j_{\mfm S}(S)$. However, $H^j_{\mfm S}(S) \simeq H^j_\mfn(S)$ as $S[F]$-modules, so we see that $$\gfdp_J(R)=\gfdp_J(S) = \gfdp_{JS}(S),$$ concluding the proof.
\end{proof}

It is also natural to ask how $\gfdp_J(R)$ behaves locally. The following local duality argument explicates the relationship between the annihilators $\mfb_j(R)$ for $0 \le j \le \dim(R)$ and $\mfb_i(R_\mfp)$ for $0 \le i \le \on{ht}(\mfp)$. Let $\bullet^{\vee_\fp}$ be the usual Matlis duality functor over the local ring $R_\fp$, and $\bullet^\vee$ denote $\bullet^{\vee_\fm}$ if $(R,\fm)$ is a local ring.

\begin{lemma}\label{lem: global-to-local vanishing for Frobenius actions}
Suppose $R$ is complete, $F$-finite, and equidimensional, and for $\mfp \in \spec(R)$ let $\mcw_\mfp$ be the composite functor $\mcw_\mfp(\bullet) = (\bullet^\vee \otimes_R R_\mfp)^{\vee_\mfp}$, viewed as a functor from the category of artinian $R[F]$-modules to the category of artinian $R_\mfp[F]$-modules. For any $x \in R$, and $e \gg 0$ we have that $x^{p^e}F^e \colon H^j_\mfm(R)\rightarrow H^j_\mfm(R)$ vanishes under $\mcw_\mfp$ if and only if $\frac{x}{1} \in \mfb_{j-d+\on{ht}(\mfp)}(R_\mfp)$. In particular, $\mfb_j(R)R_\mfp = \mfb_{j-d+\on{ht}(\mfp)}(R_\mfp)$. 
\end{lemma}

\begin{proof}
Setting $w(j)=j-(d-\on{ht}(\mfp))$, under $\mcw_\mfp$ the $p^e$-linear map $x^{p^e}F^e \colon H^j_\mfm(R)\rightarrow H^j_\mfm(R)$ is sent to the $p^e$-linear map $\frac{x^{p^e}}{1}F^e\colon H^{w(j)}_{\mfp R_\mfp}(R_\mfp) \rightarrow H^{w(j)}_{\mfp R_\mfp}(R_\mfp)$ which follows by \cite[Lem.~5.1]{BB11}, \cite[Prop.~3.1]{KMPS}. Since $\bullet^\vee$ and $\bullet^{\vee_\mfp}$ preserve annihilators, the result follows. 
\end{proof}

As with any depth theory, one expects $\gfdp_J(R)$ to be bounded by the dimension of $R$; we show that this inequality is indeed satisfied under mild hypotheses, for example, if $R$ is an $F$-finite domain. To prove this, we use the fact that for any local ring $R$, the top local cohomology module of $R$ is not nilpotent, see \cite[Lem.~4.2]{Lyu06}. 

\begin{thm}\label{thm: upper bound on gfdp}
If $R$ is $F$-finite and equidimensional, then $\gfdp_J(R)> \dim(R)$ if and only if $J\subset\sqrt{0}$. 
\end{thm}

\begin{proof}
We first show we can reduce to the complete case to apply \Cref{lem: global-to-local vanishing for Frobenius actions}. Notice that $\widehat{R}$ is still $F$-finite, and $\widehat{R}$ is equidimensional since $R$ is universally catenary. As $\dim(R)=\dim(\widehat{R})$ and $\gfdp_J(R)=\gfdp_{J\widehat{R}}(\widehat{R})$, $\dim(R)<\gfdp_J(R)$ if and only if $\dim(\widehat{R})<\gfdp_{J\widehat{R}}(\widehat{R})$. 

By \Cref{lem: global-to-local vanishing for Frobenius actions} it now follows that $\mfb_d(R)R_\mfp = \mfb_{\on{ht}(\mfp)}(R_\mfp)$; also, $\mfb_{\on{ht}(\mfp)}(R_\mfp)$ is a proper ideal since $H^{\on{ht}(\mfp)}_{\mfp R_\mfp}(R_\mfp)$ is not nilpotent. This shows $\mfb_d(R)\subset \mfp$ for all primes $\mfp \in \spec(R)$, which happens if and only if $\mfb_d(R)\subset \sqrt{0}$. \end{proof}

We now show that $\gfdp_J(R)$ can be calculated from data contained in the Zariski open set $D(J)=\{\mfp \in \spec(R) \mid J\not\subset \mfp\}$. Our result mirrors  \cite[Lem.~2.1]{HM94} for generalized depth and improves \cite[Prop.~4.6]{KMPS} for generalized weakly $F$-nilpotent rings. 

\begin{thm}\label{thm: can check gfdp_J(R) locally on D(J)}
Suppose $R$ is complete, $F$-finite, and equidimensional. For an ideal $J\subset R$ \[
\gfdp_J (R) = \inf\{\fdp (R_\mfp) + \dim(R/\mfp) \mid \mfp \in D(J)\}.
\] In particular, $R$ is generalized weakly $F$-nilpotent with respect to $J$ if and only if $D(J)$ is inside the weakly $F$-nilpotent locus of $R$.
\end{thm}

\begin{proof}
First, as $J\subset \sqrt{0}$ if and only if $D(J)=\varnothing$, and as the infimum over the empty set is $\infty$, the conclusion follows from \Cref{thm: upper bound on gfdp} if $J\subset \sqrt{0}$. So, we may assume $J\not\subset \sqrt{0}$, and adopt the notation in \Cref{lem: global-to-local vanishing for Frobenius actions}. For convenience, we write $$s(J) := \inf\{\fdp (R_\mfp) + \dim(R/\mfp) \mid \mfp \in D(J)\}.$$

Fix a prime $\mfp \in D(J)$ and let $t=\gfdp_J(R)$. By \Cref{lem: gF-depth_J observations}, we know there is an $N\gg 0$ such that $J^N\subset \mfb_j(R)$ for all $0 \le j < t$. Consequently $J^N R_\mfp\subset\mfb_i(R_\mfp)$ for $0 \le i < t-\dim(R/\mfp)$; however, as $J^N\not\subset\mfp$, this implies that $\mfb_i(R_\mfp)=R_\mfp$ for all $0\le i < t-\dim(R/\mfp)$, or equivalently that $\fdp(R_\mfp)+\dim(R/\mfp) \ge \gfdp_J(R)$. Since $\mfp$ was arbitrary in $D(J)$, we deduce $s(J)\ge \gfdp_J(R)$.

To see the converse, suppose $s(J)=k$. Then, for all $\mfp \in D(J)$, we have $\fdp(R_\mfp)\ge k-\dim(R/\mfp)$. Thus, $b_j(R_\mfp)=R_\mfp$ for all $0 \le j<k-\dim(R/\mfp)$, so by \Cref{lem: global-to-local vanishing for Frobenius actions}, it follows that $\mfb_j(R)\not\subset\mfp$ for any $\mfp \in D(J)$ and $0 \le j < k$. This implies that $J$ is inside the radical of $\mfb_0(R)\cdots \mfb_{k-1}(R)$, i.e. $\gfdp_J(R)\ge s(J)=k$. 

The final claim follows from the fact that if $\gfdp_J(R)=d$, then $\fdp (R_\mfp) +\dim(R/\mfp)\ge d$ implies $\fdp (R_\mfp) \ge \on{ht}( \mfp)$, while $\fdp (R_\mfp) \le \on{ht}(\mfp)$ by \Cref{thm: upper bound on gfdp}. 
\end{proof}

\subsection{Comparing depth-like invariants and singularities}

To further deepen the analogy between weakly $F$-nilpotent singularities and Cohen-Macaulay singularities and their variants, in this subsection we summarize their mutual relationships in the context of the generalized $F$-depth with respect to an ideal and record observations about rings of small dimension.

Recall that a local ring $(R,\fm)$ is said to be {\em generalized Cohen-Macaulay} if the local cohomology modules $H_\fm^i(R)$ have finite length for $0 \le i < \dim(R)$. The finiteness dimension of $R$, $$\operatorname{findim}(R) := \max\{ j\in \NN \mid H_\fm^j(R) \textrm{ is finite length}\},$$ measures how far the ring is from being generalized Cohen-Macaulay.  This invariant compares to the generalized depth with respect to an ideal $J$ in a similar fashion as generalized $F$-depth compares to the generalized $F$-depth with respect to an ideal.

\begin{rmk}\label{rmk: depth-like inequalities}

In light of \cref{lem: gF-depth_J observations} part (b), we get the following comparisons.

\begin{center}
    \begin{tikzcd}
    \text{depth}\arrow[phantom]{d}[labl]{\le} \arrow[phantom]{r}{\le}& \text{finiteness dimension} \arrow[phantom]{d}[labl]{\le} \arrow[phantom]{r}{\le} & \text{generalized depth wrt } J \arrow[phantom]{d}[labl]{\le}\\
    \text{F-depth} \arrow[phantom]{r}{\le} &  \text{generalized F-depth} \arrow[phantom]{r}{\le} & \text{generalized F-depth wrt } J
    \end{tikzcd}
\end{center}

As a consequence, the classes of singularities defined in terms of when these depth-like invariants are maximal can be arranged in the following hierarchy diagram.  

\begin{center}
\begin{tikzpicture}
\node[scale=.87] at (0,0) {
\begin{tikzcd}
\text{Cohen-Macaulay} \arrow[Rightarrow]{r} \arrow[Rightarrow]{d} & \text{generalized Cohen-Macaulay} \arrow[Rightarrow]{r} \arrow[Rightarrow]{d} & \text{generalized Cohen-Macaulay wrt J} \arrow[Rightarrow]{d}  \\
\text{weakly {\em F}-nilpotent} \arrow[Rightarrow]{r} & \text{generalized weakly {\em F}-nilpotent} \arrow[Rightarrow]{r} & \text{generalized weakly {\em F}-nilpotent wrt J}
\end{tikzcd}};
\end{tikzpicture}
\end{center}

Finally, note that if $R$ is $F$-injective, all the vertical inequalities in the first of the comparison diagrams above become equalities. Hence, each vertical implication in the second diagram becomes an equivalence. For example, an $F$-injective ring is Cohen-Macaulay if and only if it is weakly $F$-nilpotent. 
\end{rmk}

We conclude this section by observing that, for rings of low dimension, strong finiteness properties for local cohomology force the presence of certain singularity types. 

\begin{thm}\label{thm:dim 2 implies gen nilp}
Suppose $R$ is $F$-finite, equidimensional, and $\dim(R)\ge 2$. Then,  $H^1_\fm(R)$ is generalized nilpotent -- in particular, if $\dim(R)=2$, then $R$ is generalized weakly $F$-nilpotent.
\end{thm}
\begin{proof} Notice that the question is unaffected by passing to $R_{\text{red}}$, as $\gfdp(R)=\gfdp(R_{\text{red}})$ by \Cref{thm: gfdp and reduction + flat local extension}, and we are only checking if $\gfdp(R)>1$. Further, we may pass to to the completion, which does not affect $0^F_{H^1_\mfm(R)}$ or its colength. Now, set $d = \dim (R)$, and we have $\mfb_1(R) R_\mfp = \mfb_{1-d+\operatorname{ht}(\mfp)}(R_\mfp)$ for all primes $\mfp \in \spec (R)$ by \Cref{lem: global-to-local vanishing for Frobenius actions}. As $d \geq 2$, $1-d + \operatorname{ht}(\mfp) < 1$ for $\mfp \neq \fm$ and so $\mfb_{1-d+\on{ht}(\mfp)}(R_\mfp) = R_\mfp$ for all primes $\mfp \in \spec^\circ(R)$. Thus $V(\mfb_1(R)) \subset \{ \fm \}$, which completes the proof.
\end{proof}

\begin{rmk}\label{rmk: low dimensional nilpotent singularities}
The corollary above helps us understand the landscape of low-dimensional nilpotent-type singularities. In particular, we can assert the following singularity types for all equidimensional, $F$-finite local rings.
\renewcommand{\arraystretch}{1.5}
\begin{center}
    \begin{tabular}{c|c|c}
        $\dim(R)$ & singularity type of $R$ & reference \\ \hline
        0 & $F$-nilpotent & \cite[Prop.~2.8(1)]{PQ19} \\
        1 & weakly $F$-nilpotent & \cite[Cor.~4.6]{Lyu06} \\
        2 & generalized weakly $F$-nilpotent & \Cref{thm:dim 2 implies gen nilp}
    \end{tabular}
\end{center}

Independently of characteristic, if we assume that $R$ is equidimensional and has a dualizing complex, similar (and well-known) results follow for the Cohen-Macaulay singularities of $R_{\text{red}}$.

\begin{center}
    \begin{tabular}{c|c}
        $\dim(R)$ & singularity type of $R_{\text{red}}$ \\ \hline
        0 & regular/field  \\
        1 & Cohen-Macaulay \\
        2 & generalized Cohen-Macaulay 
    \end{tabular}
\end{center}
\renewcommand{\arraystretch}{1}

Notice that we need to focus on $R_{\text{red}}$ in the Cohen-Macaulay case, since $\depth(R_{\text{red}})$ is generally not the same as $\depth( R)$.
\end{rmk}

The following example demonstrates that the equidimensionality assumption is necessary in both of the tables in the remark above.

\begin{example}
Let $K$ be an $F$-finite field, $T=K[[x,y,z]]$, and $R=T/(xz,yz)$. Writing $I=(x,y)$ and $J=(z)$, we have $R=T/I\cap J$, so the two-dimensional reduced ring $R$ is not equidimensional. Since $I+J=\mfm$ is the maximal ideal of $T$, we get the following short exact sequence of $T[F]$-modules: \begin{center}
    \begin{tikzcd}
    0 \arrow{r} & R \arrow{r} & T/I\oplus T/J \arrow{r} & K \arrow{r} & 0
    \end{tikzcd}
\end{center} to which we may apply local cohomology at $\mfm$, yielding the following short exact sequence of $T[F]$-modules: \begin{center}
    \begin{tikzcd}
    0 \arrow{r} & K \arrow{r} & H^1_\mfm(R) \arrow{r} & H^1_\mfm(T/I) \arrow{r} & 0.
    \end{tikzcd}
\end{center}
From here, we can see that the Frobenius action on $H^1_\mfm(R)$ is injective and $H^1_\mfm(R)$ is not of finite length, so $R$ is neither generalized weakly $F$-nilpotent nor generalized Cohen-Macaulay.
\end{example}

\begin{rmk}\label{rmk:sec3localtograded} For purpose of application in the next section, we note that the proofs of \Cref{lem: equivalent conditions for J^NM to be nilpotent}, \Cref{rmk: gF-depth_J unifies F-depth and gF-depth}, and \Cref{thm: F-depth lemma}  all hold in the graded setting as well.
\end{rmk}

\section{Nilpotent singularities of Rees algebras and associated graded rings}

We now give nilpotent analogs of theorems describing how the Cohen-Macaulay property transfers between a local ring, the Rees algebra of an ideal $I$, and the associated graded ring of $I$. By \Cref{rmk:sec3localtograded}, the critical results from \Cref{sec: Fdepth} apply to this section. The primary application is a nilpotent version of a theorem of Huneke \cite[Prop.~1.1]{HuGrI}. To establish notation, we let $(S,\mfp)$ be a local ring of dimension $d\ge 1$ and fix $I\subset S$ an ideal of positive grade, so in particular $\on{ht}(I) \ge 1$. 

The Rees algebra $\cR=S[It]$ is an $\mbn$-graded ring defined by \[[\cR]_n = I^nt^n\] viewed as a subring of the polynomial ring $S[t]$. We write $It = \oplus_{n \ge 1} [\cR]_n$ for the ideal generated by all elements of positive degree, and note that $S=\cR/It$. Further, the homogeneous ideal $I\cR = \oplus_{n\ge 0} I^{n+1}t^n$ has quotient $\cR/I\cR \simeq \cG$, the associated graded ring of $I$, defined by \[[\cG]_n = I^n/I^{n+1}.\] This description of $S$ and $\cG$ as quotients of $\cR$ gives us the two fundamental short exact sequences used to study relationships between $\cR$ and $\cG$, which we label below: \begin{center}\begin{tikzcd}[row sep=0]
0 \arrow{r} & It \arrow{r} & \cR \arrow{r} & S \arrow{r} & 0 &  (A), \\
0 \arrow{r} & I\cR \arrow{r}& \cR \arrow{r} & \cG \arrow{r} & 0 & (B).
\end{tikzcd}\end{center}

Since $\on{ht}(I)\ge 1$, it is well-known that $\dim(\cR)=d+1$ and $\dim(\cG)=d$ (see, e.g., \cite[Thm.~15.7]{Mats}), and further that $It(1)\simeq I\cR$ as $\cR$-modules. To understand $It$ and $I\cR$ as graded $\cR[F]$-modules, we restrict the Frobenius endomorphism $F \colon \cR\rightarrow \cR$ to the ideals $It$ and $I\cR$, to get graded Frobenius actions $F|_{It}=\rho \colon It\rightarrow It$ and $F|_{I\cR}=\sigma  \colon  I\cR\rightarrow I\cR$. 

We then have an $\cR$-linear map $\theta  \colon  I\cR \rightarrow It$ defined on homogeneous elements by $\theta(at^n) = at^{n+1}$, which is homogeneous of degree $1$. The map $\theta$ is easily seen to be an isomorphism which demonstrates $I\cR \simeq (It)(1)$. Furthermore, for any $e \ge 0$ we get a $\cR$-linear map $\eta_e \colon It\rightarrow I\cR$ defined on homogeneous elements by $\eta_e(at^n) = at^{n-p^e}$ if $n \ge p^e$ and $0$ if $n<p^e$, which is homogeneous of degree $-p^e$. While $\eta_e$ does have a nonzero kernel, we note that it is injective in the degrees on which it is nonzero.

For any $e \ge 0$ and $n \ge 0$, these maps fit together into the commutative diagram below, where the vertical maps are $p^e$-linear over $S$: \begin{center}
\begin{tikzcd}
\,[I\cR]_n \arrow{d}{\sigma^e}\arrow{r}{\theta} &\arrow{d}{\rho^e}\,[It]_{n+1} \\
\, [I\cR]_{np^e} & \arrow{l}{\eta_e}\,[It]_{np^e+p^e}.
\end{tikzcd}
\end{center} If $e=0$, then $\rho^e$ and $\sigma^e$ are the identity maps and $\eta_0$ is the inverse map to $\theta$. In this sense, it is important to note that while $\theta \colon I\cR\rightarrow It$ is an isomorphism of graded modules, it does not commute with the Frobenius actions in each degree. Thus, while the isomorphism $I\cR \simeq It(1)$ suffices to study depth along the short exact sequences (A) and (B), we must utilize the diagram above to understand generalized $F$-depth along the same sequences. 

\begin{lemma}\label{lem: IR and It have same gfdp_J}
For any homogeneous $x \in \cR$, the $p^e$-linear map $x\rho^e \colon H^j_\mfm(It)\rightarrow H^j_\mfm(It)$ vanishes if and only if the $p^e$-linear map $x\sigma^e \colon H^j_\mfm(I\cR) \rightarrow H^j_\mfm(I\cR)$ vanishes. In particular, $\gfdp_J(I\cR)=\gfdp_J(It)$ for all homogeneous ideals $J\subset \cR$. 
\end{lemma}

\begin{proof}
Since $It$ and $I\cR$ are isomorphic as graded modules, for any sequence $\uline{y}=y_1,\cdots,y_j$ of homogeneous elements, the homogeneous associated primes of $It/(\uline{y}) It$ and $I\cR/(\uline{y}) I\cR$ are the same. Consequently, by homogeneous prime avoidance we can select 
a sequence of homogeneous elements which is a filter regular sequence on both $I\cR$ and $It$ simultaneously. We fix such a sequence $\uline{y}=y_1,\cdots,y_j$ and let $y=y_1\cdots y_j$ and $\mfq = (\uline{y})$.

We can view $H^j_\mfm(I\cR)$ and $H^j_\mfm(It)$ as the $\mfm$-torsion submodules of $H^j_\mfq(I\cR)$ and $H^j_\mfq(It)$ by the graded Nagel-Schenzel isomorphism, \cref{thm: graded NS}. We now consider the commutative diagram below, where the columns are $p^e$-linear and the rows are $\cR$-linear: \begin{center}     \begin{tikzcd} 
        H^j_\mfq(I\cR) \arrow{r}{\theta}\arrow{d}{x\sigma^e} & H^j_\mfq(It)\arrow{d}{x\rho^e}\\ 
        H^j_\mfq(I\cR) & H^j_\mfq(It) \arrow{l}{\eta_e}. 
    \end{tikzcd}
\end{center} If $x\rho^e=0$ we must have $x\sigma^e=0$. To see the converse, if we have an $\xi \in H^j_\mfq(It)$ which is $\mfm$-torsion and $x\rho^e(\xi)\neq 0$, we may assume that $\xi$ is homogeneous. By \cref{rmk: nagel-schenzel nilpotence}, we can write $\xi=[at^n + (\uline{y}^l) It]$ and $x\rho^e([at^n + (\uline{y}^l) It]) = [xa^{p^e}t^{np^e} + (\uline{y}^{lp^e})(It)^{\fbp{p^e}}]$.

For this class to be non-vanishing, we must have $y^m xa^{p^e}t^{np^e} \not \in  (\uline{y}^{lp^e+m})(It)^{\fbp{p^e}}$ for any $m \in \mbn$. In particular, letting $b_1,\cdots,b_k$ (where notably $\deg(b_i)\ge p^e$ for each $i$) be the generators of the ideal $(\uline{y}^{lp^e+m}) (It)^{\fbp{p^e}}$, we must have in $\cR$ that $y^m xa^{p^e} t^{np^e} + \sum g_ib_i \neq 0$ for any $g_i \in \cR$. Since the degree of both terms is at least $p^e$,  we know that: \[ \eta_e\left(y^m xa^{p^e}t^{np^e} + \sum g_ib_i\right) = y^m x a^{p^e}t^{np^e-p^e} + \sum g_i \eta_e(b_i) \neq 0  \] for all $m \in \mbn$ and $g_1,\cdots,g_k \in \cR$. 

We can then easily see that $\theta^{-1}(\xi) = [at^{n-1} + (\uline{y}^l)I\cR]\in H^j_\mfq(I\cR)$ does not vanish under $x\sigma^e$, as otherwise we would have a relation $y^mxa^{p^e}t^{np^e-p^e} \in (\uline{y}^{lp^e+m})(I\cR)^{\fbp{p^e}}$ for some $m$, giving $y^m xa^{p^e}t^{np^e-p^e} = \sum g_i c_i$ where $c_i = \eta_e(b_i)$, as the generators $b_i$ of $(\uline{y}^{lp^e+m})(It)^{\fbp{p^e}}$ map to generators $c_i$ of $(\uline{y}^{lp^e+m})(I\cR)^{\fbp{p^e}}$. The final claim follows by applying \Cref{lem: equivalent conditions for J^NM to be nilpotent}.
\end{proof}

We now give a nilpotent version of Huneke's theorem, see \Cref{cor: nilpotent Huneke's thm}. The key ingredient in our proof is the following result on generalized $F$-depth with respect to an ideal, which also holds for generalized depth with respect to an ideal, with a similar proof.

\begin{thm}\label{thm: bounding gfdp_J of G}
Let $J_0\subset S$ be an ideal. For $J$ a homogeneous ideal of $\cR$ such that $[J]_0=J_0$, \[\gfdp_{J\cG}(\cG) \ge \min \{\gfdp_{J_0} (S) ,\gfdp_J (\cR) -1\}.\] In particular, if $S$ is generalized weakly $F$-nilpotent with respect to $J_0$ and $\cR$ is generalized weakly $F$-nilpotent with respect to $J$, then $\cG$ is generalized weakly $F$-nilpotent with respect to $J\cG$.
\end{thm}

\begin{proof}
Let $\gfdp_J(\cR)=r$ and $\gfdp_{J_0}(S)=s$. We repeatedly apply \cref{thm: F-depth lemma} to the fundamental short exact sequences (A) and (B) above. We see first that \[
\gfdp_{J\cG} (\cG) \ge \min\{r,\gfdp_J (I\cR) - 1\} = \min \{ r,\gfdp_J (It) -1 \}
\] by \Cref{lem: IR and It have same gfdp_J}. Using the $F$-depth lemma again, we get \[\min \{ r,\gfdp_J (It) -1 \} \ge \min \{ r,\min\{r,s+1\}-1 \}.\] From here, a case analysis for $r\ge s+1$ versus $r<s+1$ allows us to place the lower bound $\min\{s,r-1\}$, as required.

The final claim follows by observing that $\dim(\cG) = \dim(S)$ and $\dim (\cR) = \dim (S) + 1$.
\end{proof}

Since the theorem above demonstrates that generalized weak $F$-nilpotence with respect to an ideal transfers from $S$ and $\cR$ to $\cG$, we get the following corollary, see \Cref{rmk: gF-depth_J unifies F-depth and gF-depth}. 

\begin{cor}\label{cor: nilpotent Huneke's thm}
Assume the setting of \Cref{thm: bounding gfdp_J of G}. If $S$ and $\cR$ are (generalized) weakly $F$-nilpotent, so is $\cG$.
\end{cor}

We conclude this section with a brief corollary bounding the generalized $F$-depth of $\cR$ with respect to the ideal $It$ of positively-graded elements of $\cR$.

\begin{cor}\label{cor:Half of HM 3.2}
Assuming the setting of \Cref{thm: bounding gfdp_J of G}, $\gfdp_{It} (\cR) \le \gfdp_{\cG+} (\cG) +1$.
\end{cor}

\begin{proof}
Observe that $[It]_0=(0)$ and $\gfdp_{(0)} (S) = \infty$, and apply \Cref{thm: bounding gfdp_J of G}.
\end{proof}

The previous corollary is a partial analog to \cite[Prop.~3.2]{HM94}, where the authors prove a similar theorem bounding the generalized depth of $\cR$ with respect to $It$. The proof of the remaining inequality in \cite[Prop.~3.2]{HM94} uses the fact that depth drops by one when taking a quotient by a regular element, see \cite[Lem.~3.1 and its proof]{HM94}. In the next section, we investigate the analogous property for $F$-depth.

\section{F-depth elements}

In this section, we continue to let $(R,\mfm)$ denote a (graded) local ring of prime characteristic $p>0$. A powerful feature of depth is that it drops by one when modding out by a regular element. This is not necessarily the case for $F$-depth. Indeed, if $(R,\mfm)$ is a local ring with $\dim (R) \ge2$ and $\fdp(R) =1$, then for every regular element $x$ in $R$ one has $\fdp (R/xR) \ge 1$ since $H^0_{\mathfrak{m}}(R/xR)$ is nilpotent. The main goal of this section is to identify a class of regular elements $x$ for which $\fdp (R/xR) = \fdp (R) -1$.  

\begin{lemma}\label{lem: gfdp_J drops by at most 1 when modding out by x}
If $x \in R$ is a (homogeneous) regular element and $J\subset R$ is any (homogeneous) ideal, then $\gfdp_J (R/xR) \ge \gfdp_J (R) - 1$; in particular, $\fdp (R/xR) \ge \fdp (R) -1$.
\end{lemma}

\begin{proof}
Consider the short exact sequence \begin{center}
    \begin{tikzcd}
    0 \arrow{r} & R \arrow{r}{\cdot x} & R \arrow{r} & R/xR \arrow{r} & 0 
    \end{tikzcd}
\end{center} where we endow the first copy of $R$ with the Frobenius action $\rho=x^{p-1}F$ to ensure that the maps in the sequence commute with Frobenius. Then, we note that $\gfdp_J (R,\rho)\ge \gfdp_J (R,F)$, since if $J^NH^j_\mfm (R)$ vanishes under $F^e$, it must also vanish under $x^{p^e-1}F^e=\rho^e$. From \Cref{thm: F-depth lemma} it then follows that \begin{align*}
\gfdp_J (R/xR) &\ge \min\{ \gfdp_J (R),\gfdp_J (R,\rho)-1 \} \\  &\ge \min\{\gfdp_J (R),\gfdp_J (R) -1\} \\ &=\gfdp_J (R) -1,
\end{align*} as required.
\end{proof}

The following definition works in either the graded or the local setting.

\begin{defn}\label{defn: Fdepth elements}
Let $(M,\rho)$ be a finitely generated (graded) $R[F]$-module. A (homogeneous) element $x\in R$ which is regular on $M$ is an \textbf{$F$-depth element on $(M,\rho)$} if $\fdp(M/xM,\overline{\rho}) = \fdp(M,\rho) -1$.
\end{defn}

If $M$ is nilpotent, the theory is vacuous since $\fdp (M) = \fdp (M/xM) = \infty$ for each $x \in R$ in this case. In some circumstances, every regular element is an $F$-depth element.

\begin{rmk}\label{rmk: every reg elt is F-depth for wFn}
Note that if $R$ is a weakly $F$-nilpotent local ring, then every regular element $x\in R$ is an $F$-depth element on $R$. This is because $\dim (R) - 1 = \fdp (R) -1 \le \fdp (R/xR) \le \dim (R/xR) = \dim (R) - 1$, where the first inequality follows from \Cref{lem: gfdp_J drops by at most 1 when modding out by x}. In particular, if $\dim (R) = 1$, every element of $R$ is an $F$-depth element on $R$ by \Cref{rmk: low dimensional nilpotent singularities}.
\end{rmk}

We now give a sufficient condition for a regular element $x\in R$ to be an $F$-depth element on $R$ in either the graded or local case.

\begin{thm}\label{thm: existence of F-depth elements}
Let $\fdp (R)=t$ and suppose $x$ is a (homogeneous) regular element on $R$. If we have that $\fdp (R/xR)=\depth (R/xR)$, then $x$ is an $F$-depth element on $R$.
\end{thm}

\begin{proof}
 Write $s=\depth(R)$. If  $\fdp (R/xR)=\depth (R/xR)$, then $\fdp (R/xR) =s-1$, since $x$ is a regular element. As  $\sqrt{xR}=\sqrt{x^{p^e}R}$, we have $\fdp (R/xR)=\fdp (R/x^{p^e}R) =s-1$ by \Cref{lem: checking gfdp_J up to radical}, so that $f^{e'}_R(H^{s-1}_\mfm(R/x^{p^e}R))$ does not vanish for each $e, e' \in \mbn$, where $f_R$ denotes the relative Frobenius action as in \cref{dff: relative action}. Thus, by (the graded version of) \cite[Thm.~4.2]{PQ19} (outlined in \Cref{thm: graded PQ 4.2}), it follows that $t=\fdp (R)\le s$. Since $s \le t$ as well, we get $s=t$ and $\fdp (R/xR) =t-1$, as required. \end{proof}

\begin{rmk}
Let $t=\depth (R/xR)$ and retain the notation used in the previous theorem. A sufficient condition to ensure that $t=\fdp (R/xR)$ is that $F\colon H^t_\mfm(R/xR)\rightarrow H^t_\mfm(R/xR)$ is injective, which is always guaranteed if $R/xR$ is $F$-injective. In the local case, a weaker, ideal-theoretic condition is given by Quy-Shimomoto in \cite[Thm.~7.3]{QS17}.

Let $\uline{y}=y_1,\cdots,y_t$ be elements of $R$ whose image in $R/xR$ form a regular sequence on $R/xR$ and let $\mfq=(\uline{y})$. The Frobenius action on $H^t_{\mfq}(R/xR)$ is then injective if and only if $(\uline{y}^n)^F=(\uline{y}^n)$ for all $n \ge 0$ since $\uline{y}^n$ is a regular sequence, so we apply \cite[Rmk~7.2.2]{QS17}. Then, the Frobenius action on $H^t_\mfm(R/xR)$ is also injective, since by the Nagel-Schenzel isomorphism this action is the restriction of the Frobenius action on $H^t_\mfq(R/xR)$ to its $\mfm$-torsion submodule.
\end{rmk}

As noted in \Cref{rmk: every reg elt is F-depth for wFn}, every nonzerodivisor of a one-dimensional local ring is an $F$-depth element on the ring. Along a similar vein, if $\dim(T) = 2$ then $T$ is either weakly $F$-nilpotent or generalized weakly $F$-nilpotent but not weakly $F$-nilpotent by \Cref{rmk: low dimensional nilpotent singularities}; in the former case every nonzerodivisor of $T$ is an $F$-depth element on $T$, while in the latter case no element of $T$ is an $F$-depth element on $T$ by \Cref{rmk: m-primary b-ideal at f-depth obstructs existence of F-depth elts} below. Consequently, the first non-trivial example of an $F$-depth element on the ring itself occurs in dimension $\ge 3$. We next offer such an example. 

\begin{example}\label{xmp: particular F-depth element}
Let $K$ be a field of prime characteristic $p>0$ and let $S=K[x,y,z,w]$ and $I=(xz,xw,y^2z,y^2w)\subset S$. Then $R=S/I$ is a standard graded $K$-algebra of dimension two and depth one, since the simplicial complex associated to the monomial ideal $I$ is disconnected. 

Letting $\mfm$ be the homogeneous maximal ideal of $R$, by results of Takayama in \cite{TAK} we have that $H^1_\mfm(R)$ is supported in degree zero and finitely many positive degrees, and further that $[H^1_\mfm(R)]_0\simeq K$, where the Frobenius action on $[H^1_\mfm(R)]_0$ is simply the Frobenius map on $K$. This implies $\fdp(R) = \depth(R) = 1$. Consequently, any ring $T$ which has a regular element $f\in T$ such that $T/fT \simeq R$ is an $F$-depth element on $T$ by the previous theorem. For a simple example, take $u\in T=R[u]$, and then $T/uT\simeq R$.
\end{example}

We next identify a necessary condition for $x$ to be an $F$-depth element on an $R[F]$-module $(M,\rho)$ in terms of $\gfdp_{xR}(M,\rho)$.  

\begin{thm}\label{thm: necessary condition for F-depth element}
Let $M$ be a finitely generated (graded) $R[F]$-module with $\fdp(M)=t$. If $x \in R$ is a (homogeneous) $F$-depth element on $M$ (with $\deg(x)=m$), then \[x \not \in \mfb_t(M) = \sqrt{\ann_R\left(H^t_\mfm(M)/0^\rho_{H^t_\mfm(M)}\right)}.\] Equivalently, if $x$ is an $F$-depth element on $M$ we have \[\fdp (M)=\gfdp_{xR} (M) = \fdp\left(M,x^{p-1}\rho\right),\] where in the graded case we adjust to $(M(-m),x^{p-1}\rho)$ in the final equality. 
\end{thm}

\begin{proof}
We present the proof in the graded case; the local case is proven similarly. As $x$ is a regular element on $M$, we have the short exact sequence of $S[F]$-modules: \begin{center}
    \begin{tikzcd}
    0 \arrow{r} & M(-m) \arrow{r}{\cdot x} \arrow{d}{x^{p-1}\rho} & M\arrow{d}{\rho} \arrow{r} & M/xM \arrow{d}{\overline{\rho}}\arrow{r} & 0 \\
    0 \arrow{r} & M(-m) \arrow{r}{\cdot x} & M \arrow{r} & M/xM \arrow{r} & 0. 
    \end{tikzcd}
\end{center} where we will abbreviate $\rho'=x^{p-1}\rho$. Then, $\rho'\colon [M(-m)]_n \rightarrow [M(-m)]_{np}$ vanishes for each $n\in\mbz$ if and only if $x^{p-1}\rho\colon [M]_n\rightarrow [M]_{np+mp-m}$ vanishes for each $n\in \mbz$. As in \Cref{lem: gfdp_J drops by at most 1 when modding out by x}, we can see that $\fdp(M/xM,\overline{\rho}) \ge \fdp(M,\rho) -1$. Hence, if $\fdp(M(-m),\rho')>t$ then $\fdp(M/xM,\overline{\rho}) \ge t$.

We now investigate when $\fdp(M(-m),\rho')=t$. If  $\rho'^e\colon H^{t}_\mfm(M(-m))\rightarrow H^t_\mfm(M(-m))$ vanishes, then so does $x\rho'^e = x^{p^e}\rho^e$, by \Cref{lem: equivalent conditions for J^NM to be nilpotent}. Then, by \Cref{lem: gF-depth_J observations}, we have $x\in \mfb_t(M)$ which forces $\gfdp_{xR} (M)>t$. Hence, if $\fdp(M/xM)=t-1$, we have $\gfdp_{xR} (M)=t$ as well.
\end{proof}

The necessary condition in \Cref{thm: necessary condition for F-depth element} is usually not sufficient. In fact, if the ideal $\mfb_t(R)$ is too large, no element of $R$ can be an $F$-depth element on $R$. 

\begin{rmk}\label{rmk: m-primary b-ideal at f-depth obstructs existence of F-depth elts}
Let $M$ be a (graded) $R[F]$-module with $t=\fdp(M)<\dim(M)$ and suppose $\mfb_t(M)$ is $\mfm$-primary. Then, $R$ has no $F$-depth elements on $M$ since every regular element of $R$ is inside $\mfm = \mfb_t(M)$. In particular, if $\fdp (M) < \gfdp (M) \le \dim (M)$, then $R$ cannot have any $F$-depth elements on $M$.
\end{rmk}

We now provide an explicit example for which no regular element is an $F$-depth element on the ring itself.

\begin{example}\label{xmp: Fermat segre proj line has no F-depth elts}
Let $K$ be a field of characteristic $p\ge 5$ and let $A = K[x,y,z]/(x^4+y^4-z^4)$ and $B=K[a,b]$. The three-dimensional ring $R=A\# B$ is generalized weakly $F$-nilpotent but not weakly $F$-nilpotent if $p\equiv 1 \bmod 4$ by \cite[Ex.~6.6]{MM}. In particular, $\fdp(R) = 2$ and $b_2(R)=\mfm_R$, so by \Cref{rmk: m-primary b-ideal at f-depth obstructs existence of F-depth elts}, $R$ cannot have any $F$-depth elements on itself. 
\end{example}

Recall by \Cref{rmk: adjusted Frob actions}, if $(M,\rho)$ is an $R[F]$-module, then for any $x \in R$, $(M,x^{p-1}\rho)$ is another Frobenius action, which has been frequently studied due to connections to deformation problems. The following lemma compares the $\rho$-closure to the $x^{p-1}\rho$-closure of the submodule $0:_M x$.

\begin{lemma}\label{lem:pulling out colons}
Let $(M,\rho)$ be an $R[F]$-module and for a fixed $x \in R$ let $\rho' = x^{p-1}\rho$. Then, $(0:_M x)^{\rho'} = 0^\rho_M :_M x$.
\end{lemma}

\begin{proof}
Notice that $m \in (0:_M x)^{\rho'}$ if and only if there is some $e \gg 0$ with $x^{p^e-1}\rho^e(m) \in (0:_M x)$ if and only if there is some $e \gg 0$ with $x^{p^e} \rho^e(m) = \rho^e(xm)=0$, i.e. if and only if $xm \in 0^\rho_M$.
\end{proof}

The following theorem demonstrates that if $x$ is a regular element of $R$ and $\fdp(R) = t$, then $x$ is an $F$-depth element on $R$ if and only if $\left( 0:_{H^t_\mfm(R)} x\right)$ does not vanish under the adjusted action $x^{p-1}F$. 

\begin{thm}
Let $(R,\mfm)$ be a (graded) local ring and let $(M,\rho)$ be a finitely-generated (graded) $R[F]$-module with $t=\fdp(M,\rho)$. Further, suppose $x \in R$ is a (homogeneous) regular element on $M$ (with $\deg x = m$) and that $x \not \in \mfb_t(M)$. Then, $x$ is an $F$-depth element on $M$ if and only if $\left(0:_{H^t_\mfm(M)} x\right)$ is not nilpotent under the Frobenius action $x^{p-1}\rho$, where in the graded case, we adjust to $\left(0:_{H^t_\mfm(M(-m))} x\right)$. In particular, if $x^{p^e-1}\rho^e\left(\operatorname{Soc} H^t_\mfm(M) \right) \neq 0$ for all $e$, then $x$ is an $F$-depth element on $M$.
\end{thm}

\begin{proof}
Again, we present the proof in the graded case; the local case is proven similarly. From the short exact sequence associated to multiplication by $x$ on $M$, we get the following long exact sequence in local cohomology: \begin{center}
    \begin{tikzcd}
    \cdots \arrow{r} & H^{t-1}_\mfm (M) \arrow{r}{\beta} & H^{t-1}_\mfm(M/xM) \arrow{r}{\delta} & H^t_\mfm(M)(-m) \arrow{r}{\cdot x} & \cdots .
    \end{tikzcd}
\end{center} We can extract the short exact sequence in the middle, noting that $H^{t-1}_\mfm(M)$ is nilpotent under $\rho$ since $t-1<\fdp(M,\rho)$: \begin{center}
    \begin{tikzcd}
    0\arrow{r} & \im(\beta) \arrow{r} & H^{t-1}_\mfm(M/xM) \arrow{r} & \im(\delta) \arrow{r} & 0.
    \end{tikzcd}
\end{center} From here we can see that $H^{t-1}_\mfm(M/xM)$ is nilpotent under $\rho$ if and only if $\im(\delta)=\ker(\cdot x) = 0:_{H^t_\mfm(M)(-m)} x$ is nilpotent under $x^{p-1}\rho$, so we may apply \Cref{lem:pulling out colons}, establishing the first claim. Further, since $\operatorname{Soc} H^t_\mfm(M)(-m) \subset 0:_{H^t_\mfm(M)(-m)} x$, the second claim follows immediately.
\end{proof}

\begin{rmk}
We note that the sufficient condition above will fail if $x^{p^e-1} \in \mfm^{\fbp{p^e}}$ for any $e \in \NN$, as then for any $\xi \in H^t_\mfm(M)(-m)$, $x^{p^e-1}\rho^e(\xi) = \sum r_i F^e(y_i \xi)$ for some $r_i \in R$, $y_i \in \mfm$. Consequently, $x^{p^e-1}\rho^e$ will vanish on $\soc H^t_\mfm(M)(-m)$. We note the similarity of this sufficient condition to Fedder's criterion for detecting $F$-purity of a hypersurface.
\end{rmk}

\subsection{Applications of F-depth elements}

It is easy to find examples of $F$-nilpotent singularities which fail to deform \cite[Ex.~2.8(2)]{ST17}. It is shown in Polstra-Quy \cite[Sec.~4,~pp.~209]{PQ19} that weakly $F$-nilpotent singularities also do not deform in general, as if $\dim(R)\ge 2$ and $x \in R$ is a regular element, $H^0_\mfm(R/xR)$ is always nilpotent even when $H^1_\mfm(R)$ is not. That is, if $x$ is a regular element $x$ so that $R/xR$ is weakly $F$-nilpotent, it is not always true that $R$ is also weakly $F$-nilpotent. However, the weakly $F$-nilpotent property deforms along $F$-depth elements.

\begin{thm}\label{thm: wFn deforms along F-depth elements}
Let $R$ be a (graded) local ring and $x \in R$ be a (homogeneous) $F$-depth element and suppose $R/xR$ is weakly $F$-nilpotent. Then, $R$ is weakly $F$-nilpotent.
\end{thm}

\begin{proof}
Notice that $\dim(R/xR) = \dim(R) - 1 = \fdp(R) - 1$, so that $\dim(R) = \fdp(R)$.
\end{proof}

We can use the theorem above and the non-deformation of weak $F$-nilpotence to produce regular elements which are not $F$-depth elements.

\begin{example}
Let $K$ be a field of prime characteristic $p>0$ and let \[R=K[x,y,z,w]/(xz,xw,y^2z,y^2w).\] Let $T$ be the localization of $R$ at the maximal ideal $(x,y,z,w)R$. The ring $T$ is two-dimensional but not weakly $F$-nilpotent (see \Cref{xmp: particular F-depth element}). But, for any regular element $a \in T$, $T/aT$ is a one-dimensional ring, and is thus weakly $F$-nilpotent by \Cref{rmk: low dimensional nilpotent singularities}. Consequently, any regular element $a \in T$ is not an $F$-depth element.
\end{example}

\begin{rmk}
A natural question to ask is whether $F$-nilpotence deforms along $F$-depth elements as well. Unfortunately, the answer to this is no \cite[Ex.~2.8(2)]{ST17}. Their example is Cohen-Macaulay, hence every regular element is an $F$-depth element.
\end{rmk}

Our next result characterizes nilpotence of the Frobenius action on graded local cohomology in negative degrees, and can be interpreted as a nilpotent version of an analogous result for generalized depth; see \cite[Lem.~2.2]{TI89}, as rephrased in \cite[Lem.~2.5]{HM94}.

\begin{thm}\label{thm: generalized F-depth at irrelevant ideal} In the graded setting, if $H^j_\mfm(R)$ is nilpotent in all negative degrees for $0 \le j < t$, then we have $\gfdp_{R^+}(R) \ge t$. The converse holds if $R$ admits a homogeneous $F$-depth element on $R$ of positive degree. 
\end{thm}

\begin{proof}
Note that if $H^j_\mfm(R)$ is nilpotent in all negative degrees for $0 \le j < t$, then $H^j_\mfm(R)$ is only possibly non-nilpotent in degree $0$, but then $R^+ [H^j_\mfm(R)]_0$ is inside the positively graded part of $H^j_\mfm(R)$, which is nilpotent. So $R^+ H^j_\mfm(R)$ is nilpotent for $0 \le j < t$, i.e. $\gfdp_{R^+}(R) \ge t$.

For the converse, suppose $\gfdp_{R^+}(R) \ge t$ and let $x \in R$ be an $F$-depth element of positive degree. Then, we must have: \[ \fdp (R) \le \gfdp_{R^+} (R) \le \gfdp_{xR} (R) = \fdp (R) 
\] by \Cref{lem: gF-depth_J observations} and \Cref{thm: necessary condition for F-depth element}. Thus $H^{j}_{\fm}(R)$ is nilpotent for $j<t$.
\end{proof}


\begin{thebibliography}{99}

\bibitem[BB05]{BB05} M. Blickle, R. Bondu, {\em Local cohomology multiplicities in terms of \'{e}tale cohomology},  Ann. Inst. Fourier (Grenoble), 55, \textbf{7}, (2005), 2239--2256.

\bibitem[BB11]{BB11} M. Blickle, G. Bl\"ockle, {\em Cartier module: Finiteness results}, J. Reine. Angew. Math., \textbf{611}, (2011), 85--123.

\bibitem[DSMNB]{dSMNB} A. De Stefani, J. Monta\~{n}o, L. N\'{u}\~{n}ez-Betancourt, {\em Blowup algebras of determinantal ideals in prime characteristic}, arXiv preprint, arXiv:2109.00592.

\bibitem[EH08]{EH08} F. Enescu, M. Hochster, {\em The Frobenius structure of local cohomology}, Algebra and Number Theory, 2, \textbf{7}, (2008), 721--754. 

\bibitem[ILL07]{ILL07} S. B. Iyengar, G. J. Leuschke, A. Leykin, C. Miller, E. Miller, A. K. Singh, and U. Walther, {\em Twenty-four hours of local cohomology}, Grad. Stud. Math. 87, American Mathematical Society, Providence, RI, 2007.

\bibitem[HWY02a]{HWY02a} N. Hara, K.-i. Watanabe, K.-i. Yoshida, {\em  $F$-rationality of Rees algebras}, J. Algebra, \textbf{247}, (2002), no.1, 153--190.

\bibitem[HWY02b]{HWY02b} N. Hara, K.-i. Watanabe, K.-i. Yoshida, {\em Rees algebras of $F$-regular type}, J. Algebra, \textbf{247}, (2002), no. 1, 191--218.

\bibitem[HS77]{HS77} R. Hartshorne, R. Speiser, {\em Local Cohomological Dimension in Characteristic $p$}, Ann. of Math., \textbf{105}, (1977),  45--79.

\bibitem[HM94]{HM94} S. Huckaba, T. Marley, {\em Depth formulas for certain graded rings associated to an ideal}, Nagoya Math. J., \textbf{133}, (1994), 57--69.

\bibitem[Huo17]{Huo17} D.T. Huong, {\em A simple proof for a theorem of Nagel and Schenzel}, VNU Journal of Science: Mathematics-Physics, 33, \textbf{4}, (2017), 87--90.

\bibitem[Hun82]{HuGrI} C. Huneke, {\em On the associated graded ring of an ideal}, Illinois J. Math., \textbf{26}, (1982), 121--137.

\bibitem[KMPS]{KMPS} J. Kenkel, K. Maddox, T. Polstra, A. Simpson, {\em $F$-nilpotent rings and permanence properties}, arXiv preprint, arXiv: 1912.01150.

\bibitem[KK21]{KK} M. Koley, M. Kummini, {\em $F$-rationality of Rees algebras}, J. Algebra, \textbf{571}, (2021), 151--167. 

\bibitem[Lyu97]{Lyu97} G. Lyubeznik, {\em F-modules:  applications to local cohomology and $D$-modules in characteristic $p >0$}, J. Reine Angew. Math., \textbf{491}, (1997),  65--130.

\bibitem[Lyu06]{Lyu06} G. Lyubeznik {\em On the vanishing of local cohomology in characteristic $p >0$}, Compos. Math., \textbf{142}, (2006), 1, 207--221.

\bibitem[Mats86]{Mats} H. Matsumura, Commutative Ring Theory, Cambridge Studies in Advanced Mathematics \textbf{8}, Cambridge, Cambridge University Press, 1986.

\bibitem[Mad19]{Mad19} K. Maddox, {\em A sufficient condition for the finiteness of Frobenius test exponents},  Proc. Amer. Math. Soc., \textbf{147}, (2019), 5083--5092. 

\bibitem[MM]{MM} K. Maddox, L. E. Miller, {\em Generalized $F$-depth and graded nilpotent singularities}, arXiv preprint, arXiv:2101.00365.

\bibitem[MP22]{MP22} K. Maddox and V. Pandey, {\em Homological properties of pinched Veronese rings}, J. Algebra \textbf{614} (2023), 307-329

\bibitem[NS94]{NS94} U. Nagel, P. Schenzel, {\em Cohomological annihilators and Castelnuovo-Mumford regularity}, in Commutative algebra: syzygies, multiplicities, and birational algebra (South Hadley, MA, 1992), 307--328, Contemp. Math., 159, Amer. Math. Soc., Providence, RI, 1994.

\bibitem[PQ19]{PQ19} T. Polstra, P. H. Quy, {\em Nilpotence of Frobenius actions on local cohomology and Frobenius closure of ideals}, J. Algebra, \textbf{529}, (2019), 196--225.

\bibitem[Quy19]{Quy19} P. H. Quy, {\em On the uniform bound of Frobenius test exponents}, J. Algebra, \textbf{518}, (2019), 119--128.

\bibitem[QS17]{QS17} P. H. Quy, K. Shimomoto, { \em $F$-injectivity and Frobenius closure of ideals in Noetherian rings of characteristic $p > 0$}, Adv. Math., \textbf{313}, (2017), 127--166.

\bibitem[Sha06]{Sha06} R. Y. Sharp, {\em Tight closure test exponents for certain parameter ideals}, Michigan Math. J. \textbf{54} (2006), no. 2, 307--317.

\bibitem[ST17]{ST17} V. Srinivas, S. Takagi, {\em Nilpotence of Frobenius action and the Hodge filtration on local cohomology}, Adv. Math., \textbf{305}, (2017), 456--478. 

\bibitem[Tak05]{TAK} Y. Takayama {\em Combinatorial characterizations of generalized {C}ohen-{M}acaulay monomial ideals},  Bull. Math. Soc. Sci. Math. Roumanie (N.S.), \textbf{48}(96) (2005), no. 3, 327--344

\bibitem[TI89]{TI89} N.V. Trung, S. Ikeda, {\em When is the Rees algebra Cohen-Macaulay?}, Comm. Alg., \textbf{17}, (1989), 2893--2922.
\end{thebibliography}
\end{document}